\documentclass[11pt]{amsart}

\usepackage{amsthm}
\usepackage{amsmath}
\usepackage{amssymb}
\usepackage{color}

\newcommand{\rank}{\mathop\mathrm{rank}}

\newcommand{\spn}{\mathop\mathrm{span}}
\newcommand{\co}{\mathop\mathrm{co}}

\newtheorem{thm}{Theorem}[section]
\newtheorem{theorem}[thm]{Theorem}

\newtheorem{corollary}[thm]{Corollary}

\newtheorem{notation}[thm]{Notation}

\newtheorem{lemma}[thm]{Lemma}
\newtheorem{proposition}[thm]{Proposition}

\newtheorem{definition}[thm]{Definition}
\theoremstyle{remark}
\newtheorem{remark}[thm]{Remark}

\newtheorem{ex}[thm]{Example}

\newcommand{\RR}{\mathbb R}

\begin{document}

\title{Remarks on scalable frames}
\author[Casazza, De Carli, Tran]{Peter G. Casazza, Laura De Carli, and Tin T. Tran}
\address{Peter Casazza: Department of Mathematics, University
of Missouri, Columbia, MO 65211-4100}
\address{Laura De Carli and Tin Tran: Department of Mathematics,  Florida International University
	Miami, FL 33199}

\thanks{The first author was supported by
 NSF DMS 1976025}

\email{Casazzap@missouri.edu}

\begin{abstract}
This paper investigates scalable frame in $\RR^n$. We define the reduced diagram matrix of a frame and use it to classify scalability of the frame under some conditions. 
We give a new approach to the scaling problem by breaking the problem into two smaller ones, each of which is easily solved, giving a simple way to check scaling. 
Finally, we study the scalability of dual frames.
\end{abstract}

\maketitle
\section{Introduction} A Parseval frame $\mathcal{X}=\{x_i\}_{i=1}^m$ for $\RR^n$ has a property that every vector $x\in \RR^n$ can be recovered via the painless reconstruction formula: $ x= \sum_{i=1}^{m}\langle x, x_i\rangle x_i $. This  is one of the properties that Parseval frames share with  orthonormal bases    that makes  them very desirable in the applications.
When a frame is not Parseval, the reconstruction formula depends on  the  inverse of the  frame operator, which may be difficult or
impossible  to calculate.
Thus, a key question  in frame theory  is the following: given a frame $\mathcal{X}=\{x_i\}_{i=1}^m\subset \RR^n$, can the frame vectors be modified so that the resulting system forms a Parseval frame?
 Since a frame is typically designed to accommodate certain requirements of an application, this modification process should be 
done in such a way as to not change the basic properties of the system.
One  way to do so  is just by scaling each  frame vector in such a way to obtain a Parseval frame.  Frame scaling   is a noninvasive procedure, and   properties such as erasure resilience or sparse expansions are left untouched by this modification. Unfortunately frame scaling 
is also one of the most difficult problems in frame theory. Much work has been done on this problem  \cite{CC,CCH,CK,CKO,CKL,DK,KO,KOP}.  

In  this paper we  find  new necessary and sufficient conditions which ensure  the scalability of frames in $\RR^n$.  Specifically, in Section 2 we  define the reduced diagram matrix of a given frame and we use it to classify the scalability of the frame under some conditions.  
In Section 3, we break out the scaling problem into two smaller ones: normalized scaling and orthogonal scaling. This gives a simpler way for checking scalability. 
Finally, in  Section 4, we present some results about the scalability of dual frames.  

\section{preliminaries}
In this section, we recall some basic facts about finite frames. For more information on the subject, see the books \cite{CG, W} and references therein. 
\begin{definition}
	A sequence of vectors $\mathcal{X}=\{x_i\}_{i=1}^m$ in $\RR^n$ is a {\bf frame} for $\RR^n$ if there are constants $0<A\leq B<\infty$ such that 
	\begin{align*}\label{eq1.1}
	A\Vert x\Vert^2\leq \sum_{i=1}^m\vert\langle x, x_i\rangle\vert^2\leq B\Vert x\Vert^2,  \ \mbox{ for all } x\in \RR^n.
	\end{align*}
\end{definition}
The constants $A$ and $B$ are called the {\it lower and upper frame bounds}, respectively. $\mathcal{X}$ is said to be a {\bf tight frame}, or an {\bf $A$-tight frame}, if $A=B$;  if $A = B = 1$, $\mathcal{X}$ is called a {\bf Parseval frame}. If all the frame elements have the same norm, we say that $\mathcal{X}$ is an {\it equal-norm frame} and if the frame elements have norm one, we call it a {\bf unit-norm frame}.  The numbers $\{\langle x, x_i\rangle\}_{i=1}^m$ are the frame coefficients of the vector $x\in \RR^n$. 

Given a frame $\mathcal{X}=\{x_i\}_{i=1}^m$ for $\RR^n$, the corresponding {\bf synthesis operator}, also denoted by $\mathcal{X}$, is the $n\times m$ matrix whose the $i$-th column is $x_i$. The adjoint matrix $\mathcal{X}^*$ is called the {\bf analysis operator}, and the {\bf frame operator} of $\mathcal{X}$ is then $S:=\mathcal{X}\mathcal{X}^*$. It is known that $S$ is a 
positive, self-adjoint, invertible operator  and satisfies:
$ Sx=\sum_{i=1}^m\langle x, x_i\rangle x_i$ for all $x\in \RR^n$.  
We recover vectors by the formula:
$$ x=S^{-1}Sx=\sum_{i=1}^m\langle x ,x_i\rangle S^{-1}x_i=\sum_{i=1}^m\langle x,S^{-1/2}x_i\rangle S^{-1/2}x_i $$
from which  follows that $\{S^{-1/2}x_i\}_{i=1}^m$ is a Parseval frame for $\RR^n$.

If the frame is $A$-tight, then its frame operator is a multiple of the identity operator. In this case, we have the following useful reconstruction formula:
\[x=\frac{1}{A}\sum_{i=1}^{m}\langle x, x_i\rangle x_i, \ \mbox{ for all } x\in \RR^n.\]
An important characterization of tight frames is using frame potentials.
\begin{definition}
	Let $\mathcal{X}=\{x_i\}_{i=1}^m$ be a collection of vectors in $\RR^n$. The {\bf frame potential} for $\mathcal{X}$ is the quantity
	\[FP(\mathcal{X})=\sum_{i=1}^{m}\sum_{j=1}^{m}|\langle x_i, x_j\rangle|^2.\]
\end{definition}
\begin{theorem}[\cite{BF}]\label{potential}
	Let $m\geq n$. If $\mathcal{X}=\{x_i\}_{i=1}^m$ is any set of unit vectors in $\RR^n$, then
	\[FP(\mathcal{X})\geq \frac{m^2}{n}\] with equality if and only if $\mathcal{X}$ is a tight frame.
\end{theorem}

The following characterization of tight frames is also well known \cite{CG}.
\begin{theorem}\label{classify tight} 
A frame $\mathcal{X}$ for $\RR^n$ is tight if and only if the row vectors of its synthesis matrix are pairwise orthogonal and have the same norm.
\end{theorem}

Let us recall the definition of scalable frames from \cite{KOP}.

\begin{definition}
A frame  ${\mathcal X}=\{x_i\}_{i=1}^m$ for $\RR^n$ is
{\bf scalable} if there exist non-negative constants $\{a_i\}_{i=1}^m$ for which  $\{a_i x_i\}_{i=1}^m$ is a tight frame.
If all the $a_i$'s are positive,  we say that the frame is {\bf strictly scalable}.
 \end{definition}

Recently, a weakening of frame scaling  called {\bf piecewise scaling} was introduced in \cite{CDT}.

 \begin{definition}\label{D1}
		  A frame $\mathcal{X}=\{x_i\}_{i=1}^m$ for $\RR^n$ is {\bf   piecewise
		  	scalable}  if  there exists an orthogonal  projection $P:\RR^n\to\RR^n$ and   constants $\{a_i, b_i\}_{i=1}^m$
			 so that $\{a_iPx_i+b_i(I -P)x_i\}_{i=1}^m$ is a Parseval frame for $\RR^n$.  
	\end{definition} 
  Here, $I=I_n$ denotes the identity operator in $\RR^n$. This approach has several advantages over regular scaling. First, very few frames are scalable while many more are piecewise scalable.
Also, one problem with scaling is that in most cases, in order to scale a frame one has to chose most of the constants equal to zero, making the scaled frame unusable in practice. Piecewise
scaling can overcome this difficulty, and in some cases no constant is required to    be  zero.

Throughout, for any natural number $m$, we use the notation $[m]$ to denote the set $[m]:=\{1, 2, \ldots, m\}$, and we write $x=(x(1), x(2), \ldots, x(n))$ to represent the coordinates of a vector $x\in \RR^n$.

\section{Reduced diagram matrices and scaling problem}

		Given a vector $x\in \RR^n$, the {\it diagram vector} $\tilde x$ of $x$ is defined in \cite{CKL} as follows.  
		$$ \tilde x:= \frac{1}{\sqrt{n-1}} \left(\begin{matrix} (x(1))^2 - (x(2))^2  \cr \vdots \cr (x(n-1))^2 - (x(n))^2  \cr \sqrt{2n}\, x(1)x(2)\cr \vdots \cr \sqrt{2n} x(n-1) x(n)\end{matrix}\right) \in \RR^{n(n-1)},
		$$
		where the difference of squares $(x(i))^2 - (x(j))^2$ and the product $x(i)x(j)$ occur exactly once for $1\leq i <j\leq n$.

		The normalization of the components of the diagram vector is chosen to preserve unit vectors.  The following result appeared in \cite{CKL}. 

\begin{proposition}\label{Prop1} For any $x, y \in \RR^n$, we have that 
$$(n-1)\langle \tilde x, \tilde y\rangle = n \langle  x, y\rangle^2-\|x\|^2 \|y\|^2.$$
\end{proposition}

		We now give an another property of the diagram vectors.
\begin{proposition}
				Let $\{x_i\}_{i=1}^m$ be any set of unit vectors in $\mathbb{R}^n$ $(m\geq n)$. Let $\tilde{x_i}$ be the diagram vector of $x_i, i\in [m]$. Then 
				\[\sum_{i, j=1}^m\langle \tilde{x_i}, \tilde{x_j}\rangle\geq 0,\] with equality if and only if $\{x_i\}_{i=1}^m$ is a tight frame.
			\end{proposition}
			\begin{proof}
				By Proposition \ref{Prop1}, for any $i, j \in [m]$, we have 
				\[(n-1)\langle \tilde{x_i}, \tilde{x_j}\rangle=n|\langle x_i, x_j\rangle|^2-1.\]
			By Theorem \ref{potential}, we have that
				\[
				(n-1)\sum_{i, j=1}^{m}\langle \tilde{x_i}, \tilde{x_j}\rangle=n\sum_{i, j=1}^{m}|\langle x_i, x_j\rangle|^2-m^2
				\geq n\frac{m^2}{n}-m^2
				=0.
				\]
The claim follows.
			\end{proof}
\begin{definition} Given a set   $\mathcal{X}=\{x_i\}_{i=1}^m$ in $\RR^n$, the {\it diagram matrix} $\theta_{\mathcal{X}}$ is the $n(n-1)\times m$ matrix 
whose  the $i$-th column is the diagram vector of  $x_i$.   
\end{definition}

		In \cite{CKL}, the authors use the diagram vectors to classify scalable frame. In the following, we will show that this can be done by using  reduced diagram vectors. Before giving the definition, we prove the following lemma.

		\begin{lemma}\label{L-rank} The rank of the diagram matrix $\theta_{\mathcal{X}}$ of a set of vectors $\mathcal{X}=\{x_i\}_{i=1}^m\subset \RR^n$ is $\leq  \frac{(n-1)(n+2)}{2}$.
		\end{lemma}
		
		\begin{proof}  The rows $n, n+1,\ldots, \frac{ n(n-1)}{2}$   of the matrix $\theta_{\mathcal{X}}$ can be obtained from   the  first $n-1$ rows. Specifically, if 
	 $1< i <j\leq n$, the rows whose elements are  
$$\left((x_1(i))^2-(x_1(j))^2, (x_2(i))^2-(x_2(j))^2, \ldots, (x_m(i))^2-(x_m(j))^2\right)$$	 is  the difference  of the  $(j-1)$-th and the $(i-1)$-th rows of    $\theta_{\mathcal{X}}$. Thus,  the rank of  $\theta_{\mathcal{X}}$ is $\leq \frac{n(n-1)}{2}+ (n-1)=\frac{(n-1)(n+2)}{2}$.
		\end{proof}

\begin{definition}  
The  {\it reduced diagram matrix} of a set of vectors $\mathcal{X}=\{x_i\}_{i=1}^m$ in $\RR^n$ is the matrix $\tilde\theta_\mathcal{X}$  obtained  by removing   the rows $n, n+1, \ldots, \frac{n(n-1)}{2}$ from the diagram matrix   $\theta_\mathcal{X}$.   
The {\it reduced diagram vector} of $x_i$, that we still denote with $\tilde x_i$, is the $i$-th column of the reduced diagram matrix $\tilde\theta_\mathcal{X}$.  
\end{definition}

By Lemma \ref{L-rank}, the rank of $\tilde\theta_\mathcal{X}$  is 
$\leq \min\left\{ \frac{(n-1)(n+2)}{2}, m\right\}$.

\medskip
We are now ready to present the following important classification of scalable frames. The main idea of this theorem is not new (see \cite{CKL, KOP}), but we will restate it here with our notation and we will provide several equivalent conditions which we will use later.

			\begin{theorem}\label{thm-classify}
				Let $\mathcal{X}=\{x_i\}_{i=1}^m$ be a frame for $\mathbb{R}^n$ and $\tilde{\theta}_{\mathcal{X}}$ be the reduced diagram matrix of $\mathcal{X}$. Then the following are equivalent.
				\begin{enumerate}
					\item $\mathcal{X}$ is scalable.
					\item There is a non-negative, non-zero vector $c = (c_1, c_2, \ldots, c_m)$ such that $c$ is orthogonal to every row of ${\tilde \theta}_{\mathcal{X}}$.
					\item 	There is a non-negative, non-zero vector $c = (c_1, c_2, \ldots, c_m)$ in the null space of ${\tilde \theta}_{\mathcal{X}}$.
					\item There is a non-negative, non-zero vector $c = (c_1,c_2, \ldots, c_m)$ such that 
					\[\sum_{i=1}^{m}c_i\tilde{x}_i=0\] i.e., 0 can be represented as a linear combination of the reduced diagram vectors $\tilde{x}_i$ with non-negative coefficients.
					\item There is a non-negative, non-zero vector $c = (c_1,c_2, \ldots, a_m)$ in the null space of the Gram matrix of ${\tilde \theta}_{\mathcal{X}}$.
				\end{enumerate}
			\end{theorem} 
\begin{proof}
It is clear that $(2) \Leftrightarrow (3) \Leftrightarrow (4)$.

$(1) \Leftrightarrow (2)$: By definition, $\mathcal{X}$ is scalable if there exist non-negative scalars $a=\{a_i\}_{i=1}^m$ so that $a\mathcal{X}:=\{a_ix_i\}_{i=1}^m$ is a tight frame. By Proposition \ref{classify tight}, this is equivalent to the fact that the rows of the synthesis matrix $a\mathcal{X}$ are orthogonal and have the same norm. Let $R_j$ denote the $j$-th row of the synthesis matrix $a\mathcal{X}$ and let $\tilde R_{j}$ denote the $j$-th row of the reduced diagram matrix of $a\mathcal{X}$.
			The rows $R_j$ and $R_1$ have the same norm if and only if \[\sum_{i=1}^m[a_ix_i(j)]^2= \sum_{i=1}^m[a_ix_i(1)]^2,  j=2, \ldots, n.\] But this is equivalent to $\langle c, \tilde{R}_j\rangle =0$ for all $j\in [n-1]$, where $c=(a_1^2,  \ldots, a_m^2)$. A similar argument shows that the rows of the synthesis matrix $a\mathcal{X}$ are orthogonal if and only if $\langle c, \tilde{R}_j\rangle =0$ for all $j\geq n$. 

$(3) \Leftrightarrow (5)$: Let $G_\mathcal{X}$ be the Gram matrix of ${\tilde \theta}_{\mathcal{X}}$. If there exists a non-negative, non-zero vector $c=(c_1, \ldots, c_m)$ so that ${\tilde \theta}_{\mathcal{X}}c=0$, then $G_\mathcal{X}x={{\tilde \theta}_{\mathcal{X}}}^*{\tilde \theta}_{\mathcal{X}}c=0$. This shows that (3) implies (5). Conversely, if ${{\tilde \theta}_{\mathcal{X}}}^*{\tilde \theta}_{\mathcal{X}}c=0$, then $\|\tilde\theta_{\mathcal{X}}c\|^2= \langle \tilde\theta_{\mathcal{X}}c, \tilde\theta_{\mathcal{X}}c\rangle=\langle\tilde\theta_{\mathcal{X}}^*\tilde\theta_{\mathcal{X}}c, c\rangle=0$. So $\tilde\theta_{\mathcal{X}}c=0$.
\end{proof}

\begin{remark} The statement ``$(1) \Leftrightarrow (5)$'' was stated in \cite{CKL} for the diagram matrix of $\mathcal{X}$. Also, the statement ``$(1) \Leftrightarrow (3)$'' appeared in \cite{KOP} when they classified $k$-scalable and strictly $k$-scalable frames. Recall  that a frame $\{x_i\}_{i=1}^m$ for $\RR^n$ is said to be $k$-scalable, respectively, strictly $k$-scalable if there is a set $I\subset [m], |I|=k$ such that $\{x_i\}_{i\in I}$ is a scalable frame, respectively, a strictly scalable frame for $\RR^n$.

\end{remark}

The following is a direct consequence of Theorem \ref{thm-classify}.
\begin{corollary}\label{C-sign}
A frame $\mathcal{X}$ is non-scalable if one of the following conditions holds.
\begin{enumerate}
\item $\tilde \theta_{\mathcal{X}}$ contains at least one row whose elements are all non-negative or all non-positive, but not all zero. 
\item The columns of $\tilde \theta_{\mathcal{X}}$ are linearly independent.
\end{enumerate}
\end{corollary}
Condition (1) of the corollary is verified, for instance, when the synthesis matix $\mathcal{X}$ contains two rows with the entries which are all positive or all negative.

The following result also appeared in \cite{KOP}, with a differrent notation. For the completeness of the paper, we will restate it here without its proof.


		 Recall that the {\it convex hull} of a set $\mathcal{X}=\{x_i\}_{i=1}^m$ in $\mathbb{R}^n$ is the set
		\[\co(\mathcal{X})=\left\{\sum_{i=1}^m\alpha_ix_i: \alpha_i\geq 0, \sum_{i=1}^{m}\alpha_i=1\right\}.\]
		
		\begin{theorem}[see also in \cite{KOP}] Let $\mathcal{X}=\{x_i\}_{i=1}^m$ be a frame for $\mathbb{R}^n$ and $\tilde{\theta}_{\mathcal{X}}$ be the reduced diagram matrix of $\mathcal{X}$. The following are equivalent.
			\begin{enumerate}
				\item $\mathcal{X}$ is scalable.
				\item $0\in \co(\tilde{\theta}_{\mathcal{X}})$ (the convex hull of the column vectors of $\tilde{\theta}_{\mathcal{X}}$).
                                     \item There is no $y\in \mathbb{R}^{\frac{(n-1)(n+2)}{2}}$ such that $\langle \tilde{x}_i, y\rangle>0$ for all $ i\in [m]$.
			\end{enumerate}
		\end{theorem}

Before we state our next theorem, we recall some standard linear algebra results. 
	
If $A=\{a_{i,j}\}_{ i,j\in [m]}$  is a square matrix, the determinant of the $(m-1)\times (m-1)$ sub-matrix obtained from $A$ by deleting the $i$-th row and $j$-th column is called the {\it minor of $a_{i,j}$}. When there is no ambiguity, this number is often denoted by $M_{i,j}$. The {\it cofactor of $a_{i,j}$} is obtained by multiplying  $M_{i,j}$ by $(-1)^{i+j}$ and can be denoted by $A_{i,j}$. 
	
Note that $\det(A)=\sum_{j=1}^m a_{i,j}A_{i,j}$,  for every $i\in [m]$. 
The following result is well known, but we prove it here for completeness.
\begin{lemma}\label{L-orth} For all $i, k\in [m]$, with  $ i \ne k$, we have that
 $\sum_{j=1}^m a_{k,j}A_{i,j}=0$. In other words, for every $i\in [m]$, the vector  $(A_{i,1}, \ldots, A_{i, m})$ is orthogonal to  all  rows of $A$ with the exception of the $i$-th row.
\end{lemma}
\begin{proof} To prove this, observe that the determinant of $A$ does not change if we replace  the   row $R_i=(a_{i,1}, \ldots, a_{i,m})$ 
 with $R_i+R_k=(a_{i,1}+  a_{k, 1}, \ldots, a_{i,m}+  a_{k, m})$, with  $k\ne i$. Thus,
$$
\det(A)=\sum_{j=1}^m a_{i,j} A_{i,j}=\sum_{j=1}^m(a_{i,j}+a_{k,j}) A_{i,j}
$$ 
and hence, $\sum_{j=1}^m a_{k,j} A_{i,j}=0$.
\end{proof}

 \begin{theorem}\label{T-main1}  Let $\mathcal{X}=\{x_i\}_{i=1}^m$ be a frame for $\RR^n$, and let $\tilde{\theta}_{\mathcal{X}}$ be the reduced diagram matrix of $\mathcal{X}$.
Assume that ${\tilde\theta}_{\mathcal{X}}$ has rank $m-1$ and let $R_1, R_2, \ldots, R_{m-1}$ be the $m-1$ linearly independent rows of $\tilde{\theta}_{\mathcal{X}}$. Let $E= (\alpha_1, \alpha_2, \ldots, \alpha_m)$ be a vector of coefficients, and let  $A$ be the matrix whose rows are $E, R_1, \ldots, R_{m-1}$. Then $\mathcal{X}$ is scalable if and only if the cofactors of the $\alpha_j$'s in the expansion of  $\det(A)$ are all non-negative or non-positive.  
\end{theorem}
\begin{proof}
Let $c=(c_1, \ldots, c_m)$ be the vector of the cofactors of the $(\alpha_1,\ldots, \alpha_m)$. Then $c\not=0$ since $\rank\tilde\theta_{\mathcal{X}}=m-1$. By Lemma \ref{L-orth}, $c$ is orthogonal to all the rows of $\tilde\theta_{\mathcal{X}}$. If $c_i's$ are all non-negative or non-positive, then $\mathcal{X}$ is scalable by Theorem \ref{thm-classify}. Conversely, if $\mathcal{X}$ is scalable, then again by Theorem \ref{thm-classify}, there is a non-negative, non-zero vector $a=(a_1, \ldots, a_m)$ such that $a$ is orthogonal to all the rows of $\tilde\theta_{\mathcal{X}}$. Note that these rows span a hyperplane in $\RR^m$. Therefore, $c$ must be a multiple of $a$. This completes the proof.
\end{proof}

\begin{remark}  We can apply Gaussian eliminations to the reduced matrix $\tilde\theta_{\mathcal{X}}$ and use the $(m-1)$ linearly independent rows of the result matrix.

\end{remark}

\begin{ex} Consider a  unit-norm frame $\mathcal{X}=\{x_i\}_{i=1}^3$ of three vectors in $\RR^2$.   After rotating and reflecting the vectors about the origin, and re-indexing them if necessary, we can always assume that $x_1=(1,0)$, $x_2= (\cos \theta, \, \sin\theta)$ and $x_3= (\cos \psi, \, \sin\psi)$ with $0\leq \theta\leq \psi\leq \pi$. The diagram matrix of  $\mathcal{X}$ is the same as the reduced diagram matrix and is 
 \begin{align*}
 \tilde{\theta}_{\mathcal{X}}&= \left(\begin{matrix}  1& (x_2(1))^2-(x_2(2))^2 & (x_3(1))^2-(x_3(2))^2    
 	\cr 0 & 2\, x_2(1)x_2(2)&  2\,x_3(1)x_3(2)   \cr
 \end{matrix}\right)\\
&=
 \left(\begin{matrix}  1& \cos(2\theta) & \cos(2\psi) \cr 0 &    \sin(2\theta) &  \sin(2\psi) \end{matrix}\right).
 \end{align*}
 $\tilde{\theta}_{\mathcal{X}}$ has maximum rank if and only if at least one of the angles $\theta, \psi$, and $\psi-\theta$ is different from $0$, $\pi/2$, and $\pi$.  The vector of the cofactors of the elements $\alpha_j$'s  in the  matrix $$
 \left(\begin{matrix} \alpha_1& \alpha_2 & \alpha_3\cr 1& \cos(2\theta) & \cos(2\psi)  \cr 0 &   \sin(2\theta) & \sin(2\psi) \end{matrix}\right)
 $$ is  
 $c=(\sin(2(\psi-\theta)), -\sin(2\psi), \sin(2\theta)).$ 
 We can check that the components of $c$ cannot be all non-positive. Moreover, they are all non-negative if and only if $0\leq \theta\leq\pi/2$ and $\pi/2\leq\psi\leq\theta+\pi/2$. Thus, the given frame is scalable if and only if the angles $\theta$ and $\psi$ satisfy these conditions. The scaling that makes the frame tight is $(\sqrt{\sin(2(\psi-\theta))}, \sqrt{-\sin(2\psi)}, \sqrt{\sin(2\theta)})$.

\end{ex}
\begin{remark} The conditions $0\leq \theta\leq\pi/2$ and $\pi/2\leq\psi\leq\theta+\pi/2$ are compatible with the fact that a frame is scalable if and only if the vectors do not lie in the same open quadrant, see \cite{KO}.
\end{remark}

In the following, we will consider the case when the orthogonal complement of the row space of the reduced diagram matrix has dimension two.

\begin{theorem} \label{T-codim2}
 	Let  $\mathcal{X}=\{x_i\}_{i=1}^m$ be a frame for $\RR^n$.
 	Assume that the reduced diagram matrix  ${\tilde \theta}_{\mathcal{X}}$ has rank $(m-2)$. Let $R_1, \ldots, R_{m-2}$ be the $(m-2)$ linearly independent rows of the matrix  ${\tilde\theta}_{\mathcal{X}}$. Let   $w_1, w_2 \in\RR^m$ be such that  the set $\{w_1, w_2, R_1,\ldots, R_{m-2}\}$ is linearly independent.  Let $E=(\alpha_1, \ldots,  \alpha_m)$  be a vector of coefficients. Then $\mathcal{X}$ is scalable if and only if there exists $t\in [0, 2\pi)$ for which the cofactors of  the $\alpha_j$ in the expansion of the determinant of the matrix whose rows are $(E, \cos t w_1+\sin t w_2, R_1, \ldots, R_{m-2})$ are all non-positive or non-negative. 
 \end{theorem}

 \begin{proof}  By Theorem \ref{thm-classify}, $\mathcal{X}$ is scalable if and only if there exsits a non-negative, non-zero vector $(a_1, \ldots, a_m)$  in the orthogonal complement of the row space of ${\tilde\theta}_{\mathcal{X}}$.  
 	Let $\xi_j$ denote the vector of the cofactors of the $(\alpha_1,\ldots, \alpha_m)$ in the matrix $A_j=(E, w_j, R_1, \ldots, R_{m-2})$. Clearly, $\xi_j \in \left(\spn\{R_i\}_{i=1}^{m-2}\right)^\perp$  and $\xi_j\perp  w_j$, $j=1, 2$.  
 	
 	Let us prove that  $ \xi_1, \xi_2$ are linearly independent, and so they form a basis for $\left(\spn\{R_i\}_{i=1}^{m-2}\right)^\perp$.  
 	Assume that $\alpha\xi_1+\beta\xi_2=0$ for some $\alpha,\beta\in\RR$. By the multilinear property of the determinant, $\alpha\xi_1+\beta\xi_2$ is  the vector of the cofactors of  $E=(\alpha_1, \ldots, \alpha_m)$ in the matrix whose rows are $(E, \alpha w_1+\beta w_2, R_1, \ldots, R_{m-2})$.  By assumption,  $\alpha w_1+\beta w_2, R_1, \ldots, R_{m-2}$ are linearly independent whenever $(\alpha, \beta)\ne (0,0) $, and so the   vector of the cofactors of $E$ can only be zero if $\alpha=\beta=0$.

 We have proved that every non-zero vector $\xi\in \left(\spn\{R_i\}_{i=1}^{m-2}\right)^\perp$ can be written as $\xi=a\xi_1 + b\xi_2$ for some $a, b\in\RR, (a, b)\not=(0,0)$. We can let $a= \lambda \cos t$ and $b=\lambda \sin t$, with $\lambda >0$ and $t\in [0, 2\pi)$, and write 
 	$\xi=\lambda[(\cos t)  \xi_1 + (\sin t)  \xi_2]$.  Thus, $\xi$ has non-negative components if and only if the same is true of $(\cos t)  \xi_1 + (\sin t)  \xi_2$. By the multilinear property of the determinant, 
 	$(\cos t)  \xi_1 + (\sin t)  \xi_2$ is the vector of the cofactors of $E $ in the matrix  $A_t$ whose rows are $(E, (\cos t) w_1+(\sin t) w_2, R_1, \ldots, R_{m-2})$; equivalently, with the notation previously introduced, $(\cos t)  \xi_1 + (\sin t)  \xi_2$ is  the vector of the cofactors of  the $\alpha_j$'s in the expansion of $\det A_t= \cos t \det  A_1+\sin t \det A_2$. This concludes the proof of the theorem.  
 	 \end{proof}

\begin{ex}

Let $\mathcal{X}=\{x_i\}_{i=1}^4$ be a unit-norm frame of  vectors in $\RR^2$, with   
  $x_1=(1,0)$ and  
   $$
    x_2= (\cos \alpha, \sin\alpha), \ x_3=(\cos\beta, \sin\beta), \ x_4=(\cos \gamma, \sin \gamma), $$ with $0<\alpha< \frac \pi 2\leq \beta <\gamma\leq \pi$. 
	 This frame is scalable because the vectors do not lie in the same open quadrant of $\RR^2$. We now find conditions to ensure that the frame is strictly scalable.

The reduced diagram matrix is 
$$
\tilde \theta_ {\mathcal{X}}=\left(\begin{matrix} 1 & \cos(2\alpha) & \cos(2\beta)  & \cos(2\gamma) \cr 0 &   \sin(2\alpha) & \sin(2\beta)& \sin(2\gamma)  \end{matrix}\right) = \left(\begin{matrix}  R_1\cr R_2\end{matrix}\right).
$$ 
The conditions on $\alpha, \beta$, and $\gamma$ imply that  $\tilde \theta_ {\mathcal{X}}$ has rank 2. 
	We observe that the vectors $w_1= (0, 0, 1, 0)$ and $w_2= (0, 0, 0, 1)$ are linearly independent from $R_1$, $R_2$ since the determinant of the matrix whose rows are $(w_1, w_2, R_1, R_2)$ is $\sin(2\alpha)\not=0$.
By Theorem \ref{T-codim2}, the frame is strictly scalable if and only if we can find $t\in [0, 2\pi)$ for which the cofactors of the $\alpha_j$'s in the matrix
$$
\left(\begin{matrix} \alpha_1&\alpha_2&\alpha_3&\alpha_4\cr 0 &0&\cos t & \sin t \cr 1 & \cos(2\alpha) & \cos(2\beta)  & \cos(2\gamma) \cr 0 &   \sin(2\alpha) & \sin(2\beta)& \sin(2\gamma)   \end{matrix}\right) 
$$
are positive or all negative. The determinant of the matrix above is $\alpha_1 A_1+\alpha_2A_2+\alpha_3A_3+\alpha_4A_4$, where
\begin{align*}
A_1&=\sin t \sin(2\beta-2\alpha) +\cos t\sin(2\alpha-2\gamma);\\
A_2&=\cos t\sin(2\gamma)-\sin t\sin(2\beta);\\
A_3&=\sin t\sin(2\alpha);\\
A_4&=-\cos t\sin(2\alpha).
\end{align*}
Thus, $\mathcal{X}$ is strictly scalable if and only if there exists $t\in [0, 2\pi)$ such that the $A_j$'s are all positive or all negative. Since we have assumed 
	$\alpha<\frac \pi 2$, the coefficients $A_3$ and $A_4$ are both positive if $t\in (\pi/2, \pi)$.
	
	If we assume  $\pi/2<\beta<\gamma<\alpha+\pi/2$,  it is easy to verify that   $A_i>0$ for   $i=1, 2 $ and for all $t\in (\pi/2, \pi) $. So, $\mathcal{X}$ is scalable with positive scalings $c=\{\sqrt{A_i}\}_{i=1}^4$. 
	Since $t$ is arbitrary in $(\pi/2, \pi)$,  we may have infinitely many ways of choosing a scaling that  makes the given frame tight.
\end{ex}

\section{A New Approach to Scaling} 

In this section, we give a new approach to scaling problem. 
We are going to break the scaling problem into two smaller ones, each of which is easily solved giving a simple way to check scaling. Let us start with introducing new notation.


\begin{notation} We denote $$\ell_2(m)^+:=\{a=\{a_i\}_{i=1}^m\in \ell_2(m): a_i\geq 0 \mbox{ for all } i\in [m]\}.$$
If $u=(u(1), u(2), \ldots, u(m)),\ v=(v(1), v(2), \ldots, v(m))\in \ell_2(m)$, we define
		\[u\bullet v =(u(1)v(1),u(2)v(2),\ldots,u(m)v(m)) \mbox{ and } u^2=u\bullet u.
		\]
		
		If $a=(a_1,a_2,\ldots,a_m)\in \ell_2(m)$, we denote $a^2=(a_1^2,a_2^2,\ldots,a_m^2)$.
If $\mathcal{X}=\{x_i\}_{i=1}^m\subset \RR^n$ we denote 
		\[ \tilde{\mathcal{X}}=\{u_j:=(x_1(j), x_2(j),\ldots,x_m(j)): j\in [n]\},
		\]
\[\tilde{\mathcal{X}}^2=\{u^2: u\in \tilde{\mathcal{X}}\}, \mbox{ and }\]
		\[ W(\ell_2(m)^+,\tilde{\mathcal{X}}^2)= \{a\in \ell_2^+(m):\langle a, u^2\rangle=1\mbox{ for all } u \in \tilde{\mathcal{X}}\}.\]
If $Z\subset \RR^n$ we denote
	\[ Z^+=\{z\in Z:z(i)\ge 0\mbox{ for all } i\in [n]\}.\]
	
\end{notation}

\begin{definition}
	Let $\mathcal{X}=\{x_i\}_{i=1}^m$ be a frame for $\RR^n$. If $W(\ell_2(m)^+,\tilde{\mathcal{X}}^2)\not= \emptyset$, we say that $\mathcal{X}$ is {\bf normalized scalable}.
\end{definition}

\begin{remark}\label{R13}
	For $x,y\in \RR^n$,
	$\langle x,y\rangle=1$ if and only if $x=cy+dv$ for some scalars $c, d$ with $c\|y\|^2=1$ and $v\in y^{\perp}$.
\end{remark}

\begin{theorem}\label{T12}
	Let $\mathcal{X}=\{x_i\}_{i=1}^m$ be a frame for $\RR^n$ and let $\{c_j\}_{i=1}^n$ be a sequence of scalars so that $c_j\|u_j^2\|^2=1$, where $u_j=(x_1(j), x_2(j), \ldots, x_m(j))\in \tilde{\mathcal{X}}$. Then
	\[ W(\ell_2(m)^+,\tilde{\mathcal{X}}^2)=\bigcap_{j=1}^n(c_ju_j^2+(u_j^2)^{\perp})^+.\]
\end{theorem}

\begin{proof}
	This is immediate from Remark \ref{R13}.
\end{proof}

\begin{lemma}
	$\mathcal{X}=\{x_i\}_{i=1}^m\subset\RR^n $ is normalized scalable if and only if there are constants $\{a_i\}_{i=1}^m$ so that the rows of the synthesis matrix $[a_ix_i]_{i=1}^m$
	all square sum to 1.
\end{lemma}

\begin{proof}
	The square sum of the $j$-th row, $j\in [n]$, of the synthesis matrix is:
	\[ \sum_{i=1}^ma_i^2x_i(j)^2=\langle a^2,u_j^2\rangle,\] where $u_j=(x_1(j), x_2(j), \ldots, x_m(j))\in \tilde{\mathcal{X}}$ and $a=(a_1, a_2, \ldots, a_m)\in \ell_2(m)$.
	So this equals 1 for all $j\in [n]$ if and only if $a^2\in W(\ell_2(m)^+,\tilde{\mathcal{X}}^2)$.
\end{proof}

\begin{lemma}
	For any $\mathcal{X}=\{x_i\}_{i=1}^m\subset \RR^n$, the set $W(\ell_2(m)^+,\tilde{\mathcal{X}}^2)$ is
	a convex set.
\end{lemma}

\begin{proof}
Let $0<\lambda<1$ and $a, b\in W(\ell_2(m)^+,\tilde{\mathcal{X}}^2)$, then oviously $\lambda a+(1-\lambda)b \in \ell_2(m)^+$. Moreover, for any $u\in \tilde{\mathcal{X}}$, we have
\[\langle \lambda a+(1-\lambda)b, u\rangle= \lambda \langle a, u\rangle +(1-\lambda)\langle b, u\rangle=1.\]
This shows that $\lambda a+(1-\lambda)b\in W(\ell_2(m)^+, \tilde{\mathcal{X}}^2)$.
\end{proof}

\begin{ex}
	It is possible that $W(\ell_2(m)^+, \tilde{\mathcal{X}}^2)=\emptyset$. To see this, take a Hadamard matrix $[\pm1]$ in $\RR^n$ and multiply the last row by 2. Let $\mathcal{X}$ be the frame whose frame vectors are the columns of the resulting matrix and denote by $u_j$ the $j$-th row of this matrix. Then $u_j=(\pm1,\ldots,\pm 1)$ for $1\le i \le n-1$ and $u_n=(\pm 2,\pm 2, \ldots,\pm 2)$. For any $a=(a_1, a_2\ldots,a_m)\in\ell_2^+(m)$ we have 
	\[ \langle a, u_1^2\rangle = \sum_{i=1}^ma_i\cdot 1=\sum_{i=1}^ma_i,\]
	while
	\[ \langle a,u_n^2\rangle = \sum_{i=1}^m4a_i.\]
\end{ex}

\begin{theorem}
	Let $\mathcal{X}=\{x_i\}_{i=1}^m$ be a frame for $\RR^n$ with the row of its synthesis matrix $u_j=(x_1(j), x_2(j), \ldots, x_m(j))\in \tilde{\mathcal{X}}, j\in [n]$. Let $\{e_i\}_{i=1}^m$ be the standard orthonormal basis for $\ell_2(m)$. The following are equivalent.
	\begin{enumerate}
		\item $\mathcal{X}$ is normalized scalable.
		\item There is a set $J\subset [n]$ with $|J|\le \rank\{u_j^2\}_{j=1}^n$, a set $I\subset [m]$, and a vector $a\in \ell_2(m)^+$ such that if $P_I$ is the orthogonal projection of $\ell_2^m$ onto $\spn\{e_i\}_{i\in I}$, then
		$\{P_Iu_j^2\}_{j\in J}$ is a basis for $\spn\{P_Iu_j^2\}_{j=1}^n$. Moreover, for every $j\in J$, $\langle a, P_Iu_j^2\rangle =1 $,  and for every $j\notin J$, $P_Iu_j^2=\sum_{k\in J}c_kP_Iu_k^2$ with
		$\sum_{k\in J}c_k=1$.
	\end{enumerate}
\end{theorem}

\begin{proof}
	$(1)\Rightarrow (2)$: Choose $a=\{a_i\}_{i=1}^m\in \ell_2(m)^+$ so that $\langle a,u_j^2\rangle =1$ for all $j\in [n]$. Let $I=\{i\in [m]:
	a_i\not= 0\}$. Choose $J\subset [n]$ so that $\{P_Iu_j^2\}_{j\in J}$ is a basis for $\spn \{P_Iu_j^2\}_{j=1}^n$. Then 
	$\langle a,P_Iu_j^2\rangle=1$ for all $j\in J$. If $j\in J^c$ then $P_Iu_j^2=\sum_{k\in J}c_kP_Iu_k^2$ and 
	\[ 1=\langle a,P_Iu_j^2\rangle = \langle a,\sum_{k\in J}c_kP_Iu_k^2\rangle =\sum_{k\in J}c_k\langle a,P_Iu_k^2\rangle =
	\sum_{k\in J}c_k.\]

	\noindent$(2)\Rightarrow (1)$: By assumption
	there are $a=(a_1, a_2, \ldots, a_m)\in \ell_2(m)^+$ so that 
	\[\langle a, u_j^2\rangle= \langle a,P_Iu_j^2\rangle=1\mbox{ for all } j\in J.\]
	If $j\in J^c$ then $P_Iu_j^2=\sum_{k\in J}c_kP_Iu_k^2$ and so
	\[ \langle a, u_j^2\rangle =\langle a,P_Iu_j^2\rangle = \langle a, \sum_{k\in J}c_kP_Iu_k^2\rangle = \sum_{k\in J}c_k\langle a,P_Iu_k^2\rangle=
	\sum_{k\in J}c_k =1.\]
	So $\{x_i\}_{i=1}^m$ is normalized scalable.
\end{proof}

For the second part of this approach, we  use the following notation:
 

\begin{notation} Let $\mathcal{X}=\{x_i\}_{i=1}^m$ be a frame for $\RR^n$ with the rows of its synthesis matrix (represented in the standard basis $\{e_i\}_{i=1}^n$ of $\RR^n$), $u_j=(x_1(j), x_2(j),\ldots,x_m(j))$, $j\in [n]$.
		We denote 
		\[ \mathcal{X}^{\perp}=\{u_i\bullet u_j:1\le i\not= j\le n\},\]
		\[ V(\ell_2(m)^+,\mathcal{X}^{\perp})= \{a\in \ell_2(m)^+ : \langle a, u\rangle=0\mbox{ for all } u\in \mathcal{X}^\perp\}.\]
	\end{notation}
\begin{definition} A frame $\mathcal{X}=\{x_i\}_{i=1}^m$ for $\RR^n$ is called orthogonal scalable if $V(\ell_2(m)^+,\mathcal{X}^{\perp})\not=\emptyset.$
\end{definition}
\begin{remark}\label{R2} 1) If $\{x_i\}_{i=1}^m$ is a frame for $\RR^n$ written with respect to the eigenbasis of its frame operator, then $V(\ell_2(m)^+,\mathcal{X}^{\perp})=\ell_2(m)^+$. This is because the rows of the synthesis matrix are orthogonal.

2) $u_i\bullet u_j:1\le i\not= j\le n$ are the last $n(n-1)/2$ rows of the reduced diagram matrix defined in Section 2.
\end{remark}

\begin{theorem}
	If $\mathcal{X}=\{x_i\}_{i=1}^m$ is a frame for $\RR^n$, then $V(\ell_2(m)^+, \mathcal{X}^{\perp})$ is a positive cone in $\RR^n$.
\end{theorem}

\begin{proof}
	Let $a, b\in V(\ell_2(m)^+, \mathcal{X}^{\perp})$ and $\alpha, \beta \ge0$. Then $\alpha a +\beta b\geq 0$, and for all $1\le i \not= j \le n$ we have that
	\[ \langle \alpha a +\beta b,u_i\bullet u_j\rangle = \alpha\langle a,u_i\bullet u_j\rangle +\beta\langle b,u_i\bullet u_j\rangle =0.\]
	The claim is proven.
\end{proof}

Now the following theorem is obvious.

\begin{theorem}
	A frame $\mathcal{X}=\{x_i\}_{i=1}^m$ for $\RR^n$ is scalable if and only if $W(\ell_2(m)^+, \tilde{\mathcal{X}}^2)\cap V(\ell_2(m)^+, \mathcal{X}^{\perp})\not= \emptyset$. In other words, the convex
	set intersects the positive cone.
\end{theorem}

\section{Scalability of Dual Frames}
Let $\mathcal{X}=\{x_i\}_{i=1}^m$ be a frame for $\RR^n$ with the frame operator $S$. A sequence $\mathcal{Y}=\{y_i\}_{i=1}^m$ in $\RR^n$ is called a {\bf dual frame} for $\mathcal{X}$ if $\mathcal{Y}$
satisfies the reconstruction formula:
\[x=\sum_{i=1}^m\langle x, x_i\rangle y_i=\sum_{i=1}^m\langle x, y_i\rangle x_i, \mbox { for all } x \in \RR^n.\]
 If $y_i=S^{-1}x_i$, $i\in [m]$, then $\mathcal{Y}$ is called the {\bf canonical dual frame}, otherwise it is called an {\bf alternate dual frame}. 

 In this section, we will study scalabilty of the dual frames.
\begin{theorem}
	Every scalable frame has a scalable alternate dual frame.
\end{theorem}

\begin{proof}
	Let $\{x_i\}_{i=1}^m$ be a scalable frame for $\RR^n$ with scaling constants $\{a_i\}_{i=1}^m$. Then after perhaps dropping those vectors which are scaled with zero  constants, we may assume
	there are $a_i>0$ so that $\{a_ix_i\}_{i=1}^m$ is a Parseval frame. So for every $x\in \RR^n$ we have
	\[ x=\sum_{i=1}^m\langle x,a_ix_i\rangle a_ix_i = \sum_{i=1}^m\langle x,x_i\rangle a_i^2x_i.\]
	It follows that $\{a_i^2x_i\}_{i=1}^m$ is a dual frame. But now, if we scale the dual frame by $\frac{1}{a_i}$ we get $\{a_ix_i\}_{i=1}^m$
	which is Parseval.
\end{proof}

\begin{corollary}
	Let $\mathcal{X}=\{x_i\}_{i=1}^m$ be a frame for $\RR^n$. The following are equivalent. 
	\begin{enumerate}
		\item $\{a_ix_i\}_{i=1}^m$ is Parseval.
		\item $\{a_i^2x_i\}_{i=1}^m$ is a dual frame for $\{x_i\}_{i=1}^m$.
	\end{enumerate}
\end{corollary}

\begin{remark}
	There are non-scalable  frames  (for example, non-orthogonal bases of $\RR^n$) whose all duals are not scalable, and also  examples of  non-scalable  frames whose  canonical dual frame is scalable.
\end{remark}


\begin{ex} Let $\mathcal{X}$ be a frame for $\RR^2$ whose the vectors are the columns of the following matrix:
\[
  \mathcal{X} =
  \left[ {\begin{array}{ccc}
    2 & 1 & 1\\
    1 & 2 & 1\\
  \end{array} } \right].
\] 
This frame is not scalable since all the vectors lie in the first open quadrant. Its frame operator $S$ is
\[
  S =\mathcal{X}\mathcal{X}^*=
  \left[ {\begin{array}{cc}
    6 & 5 \\
    5 & 6 \\
  \end{array} } \right],
\] 

We have \[
  S^{-1}=
  \dfrac{1}{11}\left[ {\begin{array}{cc}
    6 & -5 \\
    -5 & 6 \\
  \end{array} } \right],
\] and 
\[
  S^{-1}\mathcal{X}=\dfrac{1}{11}
  \left[ {\begin{array}{ccc}
    7 & -4&1 \\
    -4 & 7&1 \\
  \end{array} } \right].
\] 
The columns of the matrix $S^{-1}\mathcal{X}$ form a scalable frame since the vectors do not lie in the same open quadrant for any choices of negating the vectors. 
\end{ex}

We will provide a classification of frames whose canonical dual frame is scalable.  
First, we  consider the scalability of frames under invertible operators. 
\begin{theorem}\label{invertible}
	Let $\mathcal{X}=\{x_i\}_{i=1}^m$ be a frame for $\RR^n$ and $T$ be an invertible operator on $\RR^n$. The following are equivalent: 
\begin{enumerate}
\item $\{Tx_i\}_{i=1}^m$ is scalable with scaling $\{a_i\}_{i=1}^m$.
\item The frame operator for $\{a_ix_i\}_{i=1}^m$ is $(T^*T)^{-1}$.
\end{enumerate}
\end{theorem}
\begin{proof}
	Let $S_1$ be the frame operator for $\{a_ix_i\}_{i=1}^m$. For $x\in \RR^n$ we have
	\[ \sum_{i=1}^m \langle x,a_iTx_i\rangle a_iTx_i=T\left ( \sum_{i=1}^m\langle T^*x,a_ix_i\rangle a_ix_i\right )=TS_1T^*x.
	\]
	So $\{Tx_i\}_{i=1}^m$ is scalable with scaling $\{a_i\}_{i=1}^m$ if and only if $TS_1T^*=I$. That is, $S_1=T^{-1}(T^*)^{-1}=(T^*T)^{-1}$.
\end{proof}

\begin{corollary} Let $\mathcal{X}=\{x_i\}_{i=1}^m$ be a frame for $\RR^n$ and $T$ be an invertible operator on $\RR^n$. If $\mathcal{X}$ and $T(\mathcal{X})$ are scalable with the same scaling constants, then $T$ is unitary.
\end{corollary}
\begin{proof} This is immediate from Theorem \ref{invertible}. 
\end{proof}
The following theorem classifies when the canonical dual frame is scalable. Note that we do not require the original frame to be scalable.

\begin{theorem}\label{T3}
	Let $\mathcal{X}=\{x_i\}_{i=1}^m$ be a frame for $\RR^n$ with frame operator $S$. The following are equivalent.
	\begin{enumerate}
		\item $\{S^{-1}x_i\}_{i=1}^m$ is scalable with scaling $\{a_i\}_{i=1}^m$.
		\item There are constant $\{a_i\}_{i=1}^m$ so that the frame operator for $\{a_ix_i\}_{i=1}^n$ is $S^2$.		
\item There are constant $\{a_i\}_{i=1}^m$ so that the frame $\{a_iS^{-1/2}x_i\}_{i=1}^m$ has S as its frame operator.
	\end{enumerate}
\end{theorem}

\begin{proof}
$(1)\Leftrightarrow (2)$: This follows from Theorem \ref{invertible} with $T=S^{-1}$.
	
	
	\noindent $(2)\Rightarrow (3)$: Given (2), for every $x\in \RR^n$ we have
	\begin{align*}
		\sum_{i=1}^m\langle x,a_iS^{-1/2}x_i\rangle a_iS^{-1/2}x_i&=S^{-1/2}\sum_{i=1}^m\langle S^{-1/2}x,a_ix_i\rangle a_ix_i\\
		&= S^{-1/2}S^2(S^{-1/2}(x))\\
                  &=Sx.
	\end{align*}

	\vskip12pt
	
	\noindent $(3)\Rightarrow (2)$: Let $S_1$ be the frame operator for $\{a_ix_i\}_{i=1}^m$. Given (3), for every $x\in \RR^n$ we have
	\begin{align*}
		Sx&=\sum_{i=1}^m\langle x,a_iS^{-1/2}x_i\rangle a_iS^{-1/2}x_i\\
		&= S^{-1/2}\sum_{i=1}^m\langle S^{-1/2}x,a_ix_i\rangle a_ix_i\\
		&= S^{-1/2}S_1S^{-1/2}(x).\\
	\end{align*}
	So $S=S^{-1/2}S_1S^{-1/2}$ and hence $S_1=S^2$.
\end{proof}

\begin{remark}
	(1) Note that $S^2$ is the frame operator for the frame $\{S^{1/2}x_i\}_{i=1}^m$.

	(2) If $\{S^{-1}x_i\}_{i=1}^m$ can be scaled in more than one way, for example if  $\{a_iS^{-1}x_i\}_{i=1}^m$ and $\{b_iS^{-1}x_i\}_{i=1}^m$
	are Parseval frames, then $\{a_ix_i\}_{i=1}^m$ and $\{b_ix_i\}_{i=1}^m$ have the same frame operator.
\end{remark}

\begin{corollary} 
Let $\mathcal{X}=\{x_i\}_{i=1}^m$ be a frame for $\RR^n$ with   frame operator $S$ and let $D$ be the diagonal operator with $\{a_1, a_2, \ldots, a_m\} $ on its diagonal entries. Then 
$\{S^{-1}x_i\}_{i=1}^m$ is scalable with scaling $\{a_i\}_{i=1}^m$ if and only if $\mathcal{X}(D^2-G)\mathcal{X}^*=0$, where $G$ is the Grammian operator of $\mathcal{X}$.
\end{corollary}
\begin{proof}
We have $S=\mathcal{X X}^*$. By the Theorem \ref{T3}, $\{a_iS^{-1}x_i\}$ is Parseval if and only if $S_1=S^2$, where $S_1$ is the frame operator of $\{a_ix_i\}_{i=1}^m$. 

But $$S_1=(\mathcal{X}D)(\mathcal{X}D)^*=\mathcal{X}D^2\mathcal{X}^*$$
Thus, $S_1=S^2$ is equivalent to $\mathcal{X}D^2\mathcal{X}^*=(\mathcal{X}\mathcal{X}^*)^2=\mathcal{X}G\mathcal{X}^*$, or  $\mathcal{X}(D^2-G)\mathcal{X}^*=0$. 
\end{proof}

\begin{proposition}\label{P1}
	There are scalable frames whose canonical dual frames are not scalable.
\end{proposition}

\begin{proof}
	Let $[x_{ij}]_{i,j=1}^n$ be a unitaty Hadamard matrix (i.e., a matrix with entries $\pm 1/\sqrt{n}$) with row vectors $\{x_i\}_{i=1}^n$. Let
	\[ y_i=(x_{i,1},x_{i,2},\ldots,x_{i,n-2},2x_{i,n-1},3x_{i,n}) \mbox{ for all } i\in [n].\]
	Then $\{x_i\}_{i=1}^n \cup \{y_i\}_{i=1}^n$ is a frame for $\RR^n$. Let $S$ be its frame operator. We see that $Se_n=2e_n$ for $n=1, \ldots, n-2$, $Se_{n-1}=5e_{n-1}$, and $Se_n=10e_n$, and so  the unit vectors $e_n $ are eigenvectors of $S$  with eigenvalues $2,2,\ldots,2,5,10$.
	This frame is scalable since it contains the Parseval frame $\{x_i\}_{i=1}^n$.
	We proceed by way of contradiction by assuming the canonical dual frame is scalable. Then by Theorem \ref{T3},
	there are constants $\{a_i\}_{i=1}^n\cup \{b_i\}_{i=1}^n$ so that $\{a_ix_i\}_{i=1}^n\cup \{b_iy_i\}_{i=1}^n$ is a frame with frame operator
	$S^2$. It follows that this frame has the unit vectors $e_n $ as eigenvectors, with eigenvalues:
	\[  \sum_{i=1}^{n}a_i^2+\sum_{i=1}^nb_i^2=4, \mbox{ for the first $(n-2)$ eigenvectors}\]
	and
	\[\sum_{i=1}^na_i^2+4\sum_{i=1}^nb_i^2=
	25, \mbox{ for the $(n-1)$-th eigenvector, }\]
	and
	\[\sum_{i=1}^na_i^2+9\sum_{i=1}^nb_i^2= 100\mbox{ for the last eigenvector }.
	\]
	So 
	\[ \sum_{i=1}^na_i^2=4-\sum_{i=1}^nb_i^2,\]
	and substituting this into the next two equations yields:
	\[ 25=4-\sum_{i=1}^nb_i^2+4\sum_{i=1}^nb_i^2\mbox{ and so }\sum_{i=1}^nb_i^2=\frac{21}{3}.\]
	\[ 100=4-\sum_{i=1}^nb_i^2+9\sum_{i=1}^nb_i^2\mbox{ and so }\sum_{i=1}^nb_i^2=12.
	\]
	This contradiction completes the proof.
\end{proof}

\begin{remark} (1) Propostion \ref{P1} also shows that there are non-scalable frames whose canonical dual frames are scalable. Namely, take the canonial dual frame of the   frame in Proposition \ref{P1}.

(2)   Part (3) of Theorem \ref{T3} is interesting. Scaling means multiplying the frame vectors by a constant to get a Parseval frame.  Theorem \ref{T3}(3) says that
	we are to take the canonical Parseval frame and scale it so that the new frame has the original frame operator as its frame operator.
\end{remark}


\end{document}